\documentclass[a4paper,12pt]{article}

\makeatletter
\@addtoreset{footnote}{page}
\makeatother

\usepackage{style-enumitem}

\usepackage{amsthm}
\usepackage{amsmath,amssymb,latexsym,amsfonts,mathrsfs}

\renewcommand{\hat}{\widehat}

\newcommand{\rbra}[1]{\!\left( #1 \right)} 
\newcommand{\cbra}[1]{\!\left\{ #1 \right\}} 
\newcommand{\sbra}[1]{\!\left[ #1 \right]} 

\newcommand{\bE}{\ensuremath{\mathbb{E}}}

\newcommand{\bP}{\ensuremath{\mathbb{P}}}

\newcommand{\bR}{\ensuremath{\mathbb{R}}}

\newcommand{\cB}{\ensuremath{\mathcal{B}}}

\newcommand{\cF}{\ensuremath{\mathcal{F}}}

\theoremstyle{plain}
\newtheorem{Thm}{Theorem}[section]

\newtheorem{Lem}[Thm]{Lemma}

\theoremstyle{definition}
\newtheorem{Ass}[Thm]{Assumption}

\newtheorem{Rem}[Thm]{Remark}

\numberwithin{equation}{section}

\makeatletter
\renewcommand\section{\@startsection {section}{1}{\z@}%
                                   {-3.5ex \@plus -1ex \@minus -.2ex}%
                                   {2.3ex \@plus.2ex}%
                                   {\normalfont\large\bf}}
\makeatother

\makeatletter
\renewcommand\subsection{\@startsection {subsection}{1}{\z@}%
                                   {-3.5ex \@plus -1ex \@minus -.2ex}%
                                   {2.3ex \@plus.2ex}%
                                   {\normalfont\normalsize\bf}}
\makeatother

\usepackage{color}
\allowdisplaybreaks[1]
\usepackage{mathtools}
\mathtoolsset{showonlyrefs=true}

\begin{document}

\begin{center}
{\Large \bf 
Scale functions of space-time changed processes with no positive jumps 
}
\end{center}
\begin{center}
Kei Noba
\end{center}

\begin{abstract}
The scale functions were defined for spectrally negative L\'evy processes and other strong Markov processes with no positive jumps, and have been used to characterize their behavior. 
In particular, I defined the scale functions for standard processes with no positive jumps using the excursion measures in \cite{Nob2020_1}. 
In this paper, we consider a standard process $X$ with no positive jumps and a standard process $Y$ defined by the space-time change of $X$. 
We express the scale functions of $Y$ using the scale functions of $X$ defined in \cite{Nob2020_1} and the Volterra integral equation. 
From this result, we can express the scale functions of some important processes, such as positive or negative self-similar Markov processes with no positive jumps and continuous-state branching processes, using the scale function of spectrally negative L\'evy processes and the Volterra integral equations. 
\end{abstract}

\section{Introduction}
A $1$-dimensional L\'evy process that has no positive jumps and no monotone paths is called a spectrally negative L\'evy process. 
In the theory of spectrally L\'evy processes, the scale functions are known as a powerful tool for characterizing their behavior. 
These functions can represent the Laplace transform of the exit times from intervals of spectrally negative L\'evy processes (see, e.g., \cite[Theorem 1.2]{KuzKypRiv2012}, \cite[Theorem 8.1]{Kyp2014} or Theorem \ref{Thm206} for the case of L\'evy processes).  
Another important property is that the scale functions can characterize the potential densities of spectrally negative L\'evy processes killed on exiting intervals (see, e.g., \cite[Theorem 2.7]{KuzKypRiv2012}, \cite[Theorem 8.7]{Kyp2014} or Theorem \ref{Thm207} for the case of L\'evy processes). 
These properties allow us to obtain various results, such as the characterization of the economic cost at the moment of ruin in the risk theory (see, e.g., \cite[Section 10.1]{Kyp2014}), the derivation of optimal strategies in stochastic control problems (see, e.g., \cite{AvrPalPis2007}), and the characterization of some types of invariant measures (see, e.g., \cite{Ber1997}). 
For the above reasons, many researchers have attempted to define scale functions for other stochastic processes with no positive jumps and apply them as in the case of spectrally negative L\'evy processes. 
In particular, since the scale functions of spectrally negative L\'evy processes can be explicitly expressed using concrete parameters (see, e.g., Section \ref{Sec401}), it is important to express the scale functions of various stochastic processes using the scale functions of spectrally negative L\'evy processes.
\par
The scale functions have been defined for the following stochastic processes using the scale functions or the Laplace exponents of spectrally negative L\'evy processes.
In \cite{KypLoe2010}, the authors defined the scale functions of refracted L\'evy processes. 
As a generalization, the scale functions of level-dependent L\'evy processes were defined as solutions of the Volterra integral equations in \cite{CzaPerRolYam2019}. 
Also in \cite{NobYan2019}, the scale functions were defined for another generalization, generalized refracted L\'evy processes. 
In \cite{Pat2009}, the author defined the functions similar to the scale functions in a special case for positive self-similar Markov processes with no positive jumps.  
In \cite{LiPal2018}, the authors defined the scale functions of omega-killed spectrally negative L\'evy processes as solutions of the Volterra integral equations. 
In addition, for the scale functions of continuous-state branching processes in a special case, there is an idea written in Remark \ref{Rem402} which precedes this paper.  
\par
I would like to explain the main result, but before that, 
we focus on the results obtained with \cite{Nob2020_1} and \cite{NobYam2023+}. 
In \cite{Nob2020_1}, I defined the scale functions of standard processes with no positive jumps using the excursion measures and obtained their properties. 
In addition, in \cite{NobYam2023+}, we obtained some analytic properties of these scale functions, in particular we derived expressions for these $q$-scale functions with $q>0$ using the Volterra integral equations of the $0$-scale functions. 
I write roughly the statement of the main result of this paper, which is obtained by using the results in \cite{Nob2020_1} and \cite{NobYam2023+}. 
We consider a standard process with no positive jumps $X$ and another standard process with no positive jumps $Y$ which is defined from $X$ by time change using a continuous additive functional and by space change. 
Then, we can take the scale functions of $X$ and $Y$ as \cite{Nob2020_1} and get some properties. 
Using them and \cite{NobYam2023+}, we give the characterization of the scale functions of $Y$ via the Volterra integral equations with the scale functions of $X$.
\par
In this paper, the purpose of the above main result is to express the scale functions of some important stochastic processes that can be obtained by space-time change of spectrally negative L\'evy processes. 
Specifically, 
we will represent the scale functions of positive or negative self-similar Markov processes with no positive jumps and continuous-state branching processes via the Volterra integral equation with the scale functions of spectrally negative L\'evy processes (however, in the case of continuous-state branching processes, please pay attention to Remark \ref{Rem402}). 
\par
This paper is organized as follows. 
In Section \ref{Pre}, we recall the definition of the scale functions in \cite{Nob2020_1} and the space-time change of the processes, and decide on the notation. 
In Section \ref{Sec003}, we give the main result and its proof. 
In Section \ref{Sec004}, we give the application of the main result. Specifically, we recall the definition of the scale functions of spectrally negative L\'evy processes and give 
the expression of the scale functions of positive or negative self-similar Markov processes with no positive jumps and continuous-state branching processes explained above.

\section{Preliminaries}\label{Pre}
\subsection{The setting of standard process with no positive jumps}\label{Sec002}
Let $X:=(\Omega , \cF, \cF_t, X_t, \theta_t, \bP_x)$ be a standard process (for its definition see, e.g., (9.2) in \cite[Section I]{BluGet1968}) with state space $I$, where $I$ is an interval in $\bR$. 
Let $\partial$ be the cemetery point of $X$. 
Thus, we have $X_t(\omega)=\partial$ for $t\geq \zeta (\omega)$ with $\zeta(\omega):= \inf \{t >0: X_t(\omega)=\partial\}$. 
From now on, we will omit ``$(\omega)$'' when it is clear that ``$(\omega)$'' is used. 
We regard $\bR\cup\{\partial\}$ as the one-point compactification of $\bR$. 
We assume that $X$ has no positive jumps, i.e., for $x\in I$, it holds 
\begin{align}
X_{t-}  \geq X_t , \qquad t \in [0,\infty),\qquad \bP_x\text{-a.s.}. 
\end{align}
We define the hitting times of $X$ as follows: for $x\in I$, 
\begin{align}
\begin{aligned}
&T_x:=\inf\{t>0 : X_t=x\}, \quad 
T^-_x:=\inf\{t>0 : X_t\leq x\},\\
&\qquad \qquad \qquad T^+_x:=\inf\{t>0 : X_t\geq x\}, 
\end{aligned}
\label{1}
\end{align}
where $\inf \emptyset =\infty$. 
Note that for $x, x^\prime \in I$ with $x<x^\prime$, we have $T^+_{x^\prime}=T_{x^\prime}$, $\bP_x$-a.s. since $X$ has no positive jumps.
We impose the following assumption on $X$. 
\begin{Ass}\label{Ass201}
\begin{enumerate}
\item For $x, x^\prime \in I$ with $x<x^\prime$, it holds $\bP_x \rbra{T_{x^\prime}<\infty}>0$. 
\item The process $X$ has a reference measure $m$ i.e. $m$ is a countable sum of finite measures, 
and for $A\in \cB(I)$, $m(A)=0$ if and only if 
\begin{align}
\bE_x\sbra{\int_0^\infty 1_{A}(X_t) dt}=0,\qquad x\in I, 
\end{align}
\end{enumerate}
\end{Ass}
In addition, we assume that there exists a family of processes ${\{L^x\}}_{x \in I}$ with $L^x:=\{L^x_t: t\geq 0\}$ 
such that the map $(t, x, \omega) \mapsto L^x_t (\omega)$ is measurable with respect to ${(\cB([0, \infty))\otimes \cB(I)\otimes \cF )}^\ast$, 
the universal completion of $\cB([0, \infty))\otimes \cB(I)\otimes \cF $, and 
that $\bP_x$-a.s with $x\in I$, its holds, for $t\geq 0$ and non-negative measurable function $f$, 
\begin{align}
\int_0^t f(X_s) ds =\int_I f(x^\prime) L^{x^\prime}_t m(dx^\prime). 
\label{2}
\end{align}
We also impose assumptions on $L^x$. 
\begin{Ass}\label{Cond202a}
\begin{enumerate}
\item When $x$ is regular for itself, the process $L^x$ 
is continuous, non-decreasing, and satisfies that $L^x_t$ is $\cF_t$-measurable for $t\geq 0 $, 
\begin{align}
L^{x}_{s+t}
=L^{x}_{s}+L^{x}_{ t }\circ \theta_{ s}, \quad s, t \geq 0, \quad \bP_{x^\prime}\text{-a.s. } x^\prime\in I
\end{align}
and 
\begin{align}
\cbra{x^\prime\in I : \bE_{x^\prime}\sbra{e^{-R^x}}=1}=\{x\}, 
\end{align}
where $R^x:=\{t>0 : L^x_t>0\}$ and $e^{-\infty}:=0$. 
\item When $x$ is irregular for itself, there exists a constant $l^x\in (0, \infty)$ such that, for $x^\prime\in I$, 
\begin{align}
L^x_t=l^x \sharp \{s \in [0, t): X_s=x\}, \qquad t\geq 0 ,\quad \bP_{x^\prime}\text{-a.s.}. \label{6}
\end{align}
\end{enumerate}
\end{Ass}
For the definition of ``regular'' and ``irregular'', see, e.g., (11.1) in \cite[Section I]{BluGet1968}. 
In this paper, we call the process $L^x$ a local time at $x\in I$. 
\begin{Rem}\label{Rem203}
If the map $(x, x^\prime) \mapsto \bE_x \sbra{e^{-T_{x^\prime}}}$ 
is $\cB (I)\otimes \cB (I)$-measurable, 
then we can take the local times which satisfy the conditions above by \cite[Theorem 18.4]{GemHor1980}. 
For example, it is easy to confirm from existing results that spectrally negative L\'evy processes satisfy this condition. 
\end{Rem}
\par
Let $U$ be the set of function $u$ from $[0, \infty)$ to $I \cup\{\partial \}$ which is c\`adl\`ag and satisfies 
$u(t)=\partial$ for $t\geq \zeta^U (u)$ where $\zeta^U(u):= \inf \{t >0: u(t)=\partial\}$. 
The set $\cB(U)$ denote the class of Borel sets of $U$ equipped with the Skorokhod topology.
For $x \in I$, let $n_x$ be the measure on $U$ satisfying the following assumptions in each case above. 
\begin{Ass} \label{Cond202}
\begin{enumerate}
\item When $x$ is regular for itself and not a trap, 
the measure $n_x$ on $U$ is the same as $\hat{P}$ which was constructed in \cite[Section III.3 (e) and (g)]{Blu1992} associated with $L^x$. 
Note that the measure $n_x(T_x \in \cdot)$ is equal to the L\'evy measure of the killed subordinator $\eta^x:=\{\eta^x_t: t\geq 0 \}$ where $\eta^x_t:= \sup\{s>0:  L^x_s > t\}$. 
\item When $x$ is irregular for itself, the measure $n_x$ on $U$ has the same distribution as $\frac{1}{l^x}\bP_x^x$ where $\bP_x^x$ is the low of $\{X_{t\land T_x}: t\geq 0\}$ under $\bP_x$. 
\end{enumerate}
\end{Ass}
In this paper, we call $n_x$ an excursion measure away from $x\in I$. 
We also denote the coordinate process of $n_x$ with $x\in I$ by $X$ and use the same notation for hitting times as $\zeta$ and \eqref{1}.
When considering the excursion measures of other stochastic processes, we will use the same notations for the coordinate processes of the excursion measures and related notations as for the original stochastic processes, as in the case of $X$ above.
\begin{Rem}\label{}
By \cite[Remark 2.1]{Nob2020_1} and Assumption \ref{Ass201} (i), any points in $I$ except for the point $\sup I $ cannot be holding points or traps.
\end{Rem}

\subsection{The scale functions}\label{Sec202}
For $q\geq 0$, we define function $W_X^{(q)}$ from $I \times I $ to $[0, \infty) $ as follows: for $x, x^\prime \in I$,  
\begin{align}
W_X^{(q)}(x, x^\prime)=
\begin{cases}
\frac{1}{n_{x^\prime} \sbra{e^{-q T^+_x}}}, \qquad &x\geq x^\prime, 
\\
0, \qquad &x<x^\prime,
\end{cases}
\label{3}
\end{align}
where 
$\frac{1}{\infty}=0$
if $x^\prime$ is not a trap. 
If $\sup I \in I$ and $\sup I$ is a trap, then we define $W_X^{(q)}(x, \sup I):=0$ for $x\in I$. 
We call $W_X^{(q)}$ the $q$-scale function of $X$. 
For simplicity, we write $W_X := W_X^{(0)}$. 
\par
The following two theorems proved in \cite{Nob2020_1} are not be used in this paper, but they are important properties of the scale functions, so I write them here. 
\begin{Thm}[{\cite[Theorem 3.4]{Nob2020_1}}]\label{Thm206}
For $q\geq 0$ and $x, a, b \in I$ with $a<x<b$, we have 
\begin{align}
\bE_x \sbra{e^{-qT^+_b}; T^+_b<T^-_a}=\frac{W_X^{(q)}(x,a )}{W_X^{(q)}(b,a )}. 
\label{204}
\end{align}
\end{Thm}
By \eqref{2}, we have, for $q\geq 0$, $x, a, b \in I$ with $a<x<b$ and non-negative measurable function $f$, we have 
\begin{align}
\bE_x \sbra{\int_0^{T^-_a \land T^+_b}e^{-qt} f(X_t) dt}
=\int_{(a,b)} f(x^\prime) \bE_x\sbra{\int_{[0, T^-_a \land T^+_b]} e^{-qt}dL^{x^\prime}_t }m (dx^\prime). 
\end{align}
\begin{Thm}[{\cite[Theorem 3.6]{Nob2020_1}}]\label{Thm207}
For $q\geq 0$ and $x, x^\prime, a, b \in I$ with $a<b$ and $x, x^\prime \in (a, b)$, we have 
\begin{align}
\bE_x\sbra{\int_{[0, T^-_a \land T^+_b]} e^{-qt}dL^{x^\prime}_t }= \frac{W_X^{(q)}(x, a) }{W_X^{(q)}(b, a)} W_X^{(q)}(b, x^\prime)- W_X^{(q)}(x, x^\prime). 
\end{align}
\end{Thm}
The following theorem is important for proving the main theorem. 
\begin{Thm}[{\cite[Proposition 3.3 and Theorem 3.5]{NobYam2023+}}]\label{Thm204}
For $q>0$ and $a, b, c \in I$ with $b\leq a \leq c$, the function $W_X^{(q)}(a, \cdot)$ on $[b, c]$ is the unique solution of the Volterra integral equation
\begin{align}
f(x)=W_X(a, x)+q \int_{(x , a)} f(u) W_X(u, x) m(du), \qquad x \in [b, c] . 
\label{13}
\end{align}
\end{Thm}
\begin{Rem}
In \cite{Nob2020_1} and \cite{NobYam2023+}, the above theorems are proved under the condition that  the map $(x, x^\prime) \mapsto \bE_x \sbra{e^{-T_{x^\prime}}}$ is $\cB (I)\otimes \cB (I)$-measurable. However, this condition is used only for the existence of local times (see Remark \ref{Rem203}), and thus these theorems hold in the situation above.
\end{Rem}
\begin{Rem}
The scale functions of $X$ are defined uniquely from the local times. 
However, the setting of the reference measure $m$ is not unique, and the setting of the local times is unique in $m$-a.e.. 
The definition of the scale functions depends on them. 
For example, let $m$ be a reference measure on $I$, ${\{ L^x\}}_{x\in I}$ be the set of local times associated with $m$ and 
${\{W_X^{(q)}\}}_{q\geq 0}$ be the set of the scale functions associated with ${\{ L^x\}}_{x\in I}$.  
In addition, let $h$ and $h^\prime$ be measurable functions from $I$ for $(0, \infty)$ such that $h=h^\prime$, $m$-a.e.. 
Then, the measure $m^\prime(dx)=h(x)m(dx)$ on $I$ is also a reference measure. 
In addition, ${\{ L^{ \prime, x}\}}_{x\in I}$, where $ L^{ \prime, x}_t := \frac{1}{h^\prime(x)} L^x_t$ for $x \in I$ and $t\geq 0$, is the set of local times associated with $m^\prime$, 
and ${\{W_X^{ \prime, (q)}\}}_{q\geq 0}$, where $W_X^{ \prime, (q)}(x, x^\prime)=\frac{1}{h^\prime(x^\prime)}W_X^{ \prime, (q)}(x, x^\prime)$ for $q\geq 0 $ and $x, x^\prime \in I$, is the set of the scale functions associated with ${\{ L^{\prime,x}\}}_{x\in I}$. 
Regardless of how we fix the measure $m$ and the local times, the scale functions defined according to them satisfy the above theorems.
\end{Rem}

\subsection{Space-time changed process}\label{SubSec203}
We fix the functions $h_S$ and $h_T$ from $I$ to $\bR$ satisfying the following conditions 
\begin{enumerate}
\item The function $h_S$ is continuous and strictly increasing. 
\item The function $h_T$ is $\cB(I)$-measurable and strictly positive. 
\item We assume that $h_S(\partial):= \partial$ and $h_T(\partial):=0$. 
\item We define the process $A:=\{A_t : t\geq 0\}$ as 
\begin{align}
A_t:=\int_0^t h_T (X_t) dt,\qquad t\geq 0 .
\end{align}
\end{enumerate}
We assume that $A_t<\infty$ for $t\in[0, \zeta)$, $\bP_x$-a.s. for $x\in I$. 
The process $A$ is a continuous additive functional. 
We write its right inverse as 
\begin{align}
\tau(t):=\inf\{s>0:A_s>t\},\qquad t\geq 0 .
\end{align}
Since $A$ is strictly increasing and continuous on $[0, \zeta)$, we have $\tau(A_t)=t$ for $t\in[0, \zeta)$ and $A_{\tau(t)}=t$ for $t\in[0, A_\zeta)$. 
For $t\geq 0 $, we write $Y_t:=h_S (X_{\tau(t)})$ with $X_{\infty}:=\partial$, $\cF^Y_t:= \cF_{\tau(t)}$ and $\theta^Y_t:= \theta_{\tau(t)}$. For $y \in h_S(I)$, we write $\bP^Y_y:= \bP_{h_S^{-1} (y)}$. 
Then, 
the process $Y:=(\Omega , \cF, \cF^Y_t, Y_t, \theta^Y_t, \bP^Y_y)$ is a standard process on the interval $h_S(I)$ with no positive jumps (see, e.g., (2.11) in \cite[Section V]{BluGet1968}). 
\par
We express the hitting times of $Y$ in the same way as for $X$, by putting $Y$ on the right shoulder of the symbols for the hitting times of $X$.
\par
It is obvious that 
the measure $m(h_S^{-1}(\cdot))$ is a reference measure of $Y$. 

\section{The scale functions of space-time changed processes}\label{Sec003}
Let $X$ be the standard process defined in Section \ref{Sec002} 
and let $Y$ be the process defined as Section \ref{SubSec203}. 
In this section, we represent the scale functions of $Y$ using the scale functions of $X$ and the Volterra integral equations. 
\par
Before writing the main theorem, we need to set the situation. 
We define the reference measure $m$ of $X$, the local times ${\{L^x\}}_{x\in I}$ of $X$, 
the excursion measures ${\{n_x\}}_{x\in I}$ (except trap in $I$) of $X$ 
and the scale functions ${\{W_X^{(q)}\}}_{q\geq 0}$ of $X$ as Section \ref{Sec002}. 
For a measurable function $h_D$ from $h_S(I)$ to $(0, \infty)$, 
we define 
\begin{align}
m_{Y} (B):= \int_B h_D(y) m ( h_S^{-1} (dy) ), \qquad B\in \cB ( h_S(I) ). 
\end{align}
Then $m_Y$ is a reference measure of $Y$. 
\begin{Rem}
By changing the function $h_D$ appropriately, $m_Y$ can represent any reference measures of $Y$.
\end{Rem}
For $y \in h_S (I)$, we define the process $L^{Y, y}:=\{L^{Y, y}_t: t\geq 0 \}$ as 
\begin{align}
L^{Y, y}_t         : =  H(y) L^{h_S^{-1}(y)}_{\tau(t)} ,\qquad t\geq 0,\label{12}
\end{align}
where $H(y):=\frac{h_T(h_S^{-1}(y))}{h_D(y)} $ for $y \in h_S(I)$. 
Then, we have the following lemma. 
\begin{Lem}\label{Lem302}
The processes ${\{L^{Y, y}\}}_{y \in h_S(I)}$ are local times of $Y$ satisfying \eqref{2} with respect to $m_Y$. 
\end{Lem}
\begin{proof}
It is obvious that the map $(t, y, \omega) \mapsto H(y) L^{h_S^{-1}(y)}_{\tau(t)} (\omega)$ is measurable with respect to ${(\cB([0, \infty))\otimes \cB(h_S(I))\otimes \cF )}^\ast$. 
For $y\in h_S(I)$, non-negative measurable function $f$ on $h_S(I)$ and $t\geq 0 $, we have 
\begin{align}
\int_0^t f(Y_s) ds=\int_0^t f(h_S(X_{\tau(s)}))ds
=&\int_0^{\tau(t)} f(h_S(X_{u}))h_T(X_u)du \\
=&\int_I f(h_S(x))h_T(x) L^x_{\tau(t)} m(dx) \\
=&\int_{h_S(I)} f(y^\prime)H(y^\prime) L^{h_S^{-1}(y^\prime)}_{\tau(t)} m_Y(dy^\prime),\quad \bP^Y_y\text{-a.s.}, 
\end{align}
where in the second and the last equalities, we used the change of variables $u=\tau(s)$ and $x=h_S^{-1}(y^\prime)$, respectively, and in the third equality, we used \eqref{2}. Thus, the processes ${\{L^{Y, y}\}}_{y \in h_S(I)}$ satisfies \eqref{2} for $Y$. 
We confirm that ${\{L^{Y, y}\}}_{y \in h_S(I)}$ satisfies Assumption \ref{Cond202a}.  
\begin{enumerate}
\item We assume that $y\in h_S(I)$ is regular for itself for $Y$. 
Then, $h_S^{-1}(y)$ is also regular for itself for $X$ by the definition of $\tau$. 
Thus, $L^{Y, y}$ is a continuous, non-decreasing process. 
In addition, $L^{Y, y}_t=L^{h_S^{-1}(y)}_{\tau(t)}$ is $\cF^Y_t$-measurable for $t\geq 0$ since $\tau(t)$ is a stopping time of ${\{\cF_t\}}_{t\geq 0 }$. 
By the property of the local time $L^{h_S^{-1}(y)}$, we have, for $y^\prime \in h_S(I)$,
\begin{align}
L^{Y, y}_{s+t}=L^{h_S^{-1}(y)}_{\tau(s+t)}=L^{h_S^{-1}(y)}_{\tau(s)+ (\tau(t)) \circ \theta_{\tau(s)} }
=&L^{h_S^{-1}(y)}_{\tau(s)}+(L^{h_S^{-1}(y)}_{ \tau (t)} )\circ \theta_{ \tau(s)} \\
=&L^{Y, y}_{s}+L^{Y, y}_{t}\circ \theta^Y_s,\quad s, t \geq 0, \quad \bP^Y_{y^\prime}\text{-a.s..}
\end{align}
We write $R^{Y,y }:=\{t>0: L^{Y, y}_t>0\}$.  
Then, 
we have 
\begin{align}
\cbra{y^\prime\in h_S(I): \bE^Y_{y^\prime}\sbra{e^{-R^{Y,y}}}=1}&=\Big{\{}y^\prime \in h_S(I): \bE_{h_S^{-1}(y^\prime) }\Big{[}e^{- \tau^{-1}({R^{h_S^{-1}(y)}})}\Big{]}=1\Big{\}}\\
&=\{y\},
\end{align} 
by the definition of the local time $L^{h_S^{-1}(y)}$. 
Therefore, the process $L^{Y, x}$ satisfies Assumption \ref{Cond202a} (i). 

\item We assume that $y\in h_S^{-1}(I)$ is irregular for itself for $Y$. 
Then, $h_S^{-1}(y)$ is also irregular for itself for $X$ by the definition of $\tau$. 
Thus, by \eqref{6}, we have, for $y^\prime\in h_S(I)$, 
\begin{align}
L^{Y,y}_t=H(y) L^{h_S^{-1}(y)}_{\tau(t)}&=H(y)l^{h_S^{-1}(y)} \sharp \{s \in [0, \tau(t)): X_s={h_S^{-1}(y)}\}
\\
&=H(y)l^{h_S^{-1}(y)} \sharp \{s \in [0,t): Y_s=y\}
, \qquad t\geq 0 ,\ \ \bP^Y_{y^\prime}\text{-a.s.}. 
\label{14}
\end{align}
\end{enumerate}
The proof is complete. 
\end{proof}
\par
For $y\in h_S(I)$ which is not a trap, we denote $n^Y_y$ the excursion measure away from $y$ for $Y$ associated with the local time $L^{Y, y}$. 
In addition, we denote the functions $\{W^{(q)}_Y\}_{q \geq 0}$ from $h_S(I)\times h_S(I)$ to $[0, \infty)$ the scale functions of $Y$ defined from the excursion measures above. 
Then, we have the following main theorem. 
\begin{Thm}\label{Thm301}
For $q\geq 0$ and $a\in h_S(I)$, the function $y \mapsto W^{(q)}_Y(a, y)$ on $h_S(I)$ is an unique solution of the Volterra integral equation
\begin{align}
\begin{aligned}
f(y)=H(y)W_X(&h_S^{-1}(a), h_S^{-1}(y))
\\
&+ H(y) q\int_{(x, a)} f(y^\prime) W_X(h_S^{-1}(y^\prime), h_S^{-1}(y))  m_Y(dy^\prime), \\ 
 \end{aligned}
 \quad
 y \in h_S (I).
 \label{4}
\end{align}
\end{Thm}
\begin{Rem}
For $a, b, c \in h_S(I)$ with $b\leq a \leq c$, the function $t\mapsto W^{(q)}_Y(a, y)$ on $[b ,c]$ is a unique solution of \eqref{4} with $y\in[b, c]$. 
This fact can be easily confirmed by Theorem \ref{Thm204} and the proof of Theorem \ref{Thm301} (especially around \eqref{5}). 
\end{Rem}
\par
Theorem \ref{Thm301} can be proved immediately by using the following lemma and Theorem \ref{Thm204}. 
\begin{Lem}\label{LemA01}
For $y\in h_S(I)$ which is not a trap and non-negative functional $F$ on $U$, we have 
\begin{align}
n_{y}^Y\sbra{ F(\{ Y_{t}  : t\geq 0 \})  }=
\frac{1}{H(y)}n_{h_S^{-1}(y)}\sbra{ F(\{ h_S (X_{\tau(t)}) : t\geq 0 \})  } .\label{10}
\end{align}
\end{Lem}
\begin{proof}
We define, for 
$t\geq 0$,
\begin{align}
g_{h_S^{-1}(y)}(t):= \sup\cbra{s \leq t: X_s=h_S^{-1}(y)}, \qquad &g^Y_{y}(t):= \sup\cbra{s \leq t: Y_s=y},\\
T_{h_S^{-1}(y)}(t):=\inf\cbra{s>t : X_s=h_S^{-1}(y)}, \qquad &T^Y_{y}(t):=\inf\cbra{s>t : Y_s=y}. 
\end{align}
We divide the proof into two cases below. 
\par
(i) We assume that $y\in h_S(I)$ is regular for itself. 
By the property of Poisson point processes and the construction of the excursion point measure in \cite[Section III.3 (d)]{Blu1992}, we have, for $\varepsilon>0$,  
\begin{align}
&n_{y}^Y\sbra{ F(\{ Y_{t}  : t\geq 0 \}) ; T^Y(\varepsilon)<\infty }
\Big{\slash}n_{y}^Y\rbra{T^Y(\varepsilon)<\infty}
\\
=&\bE^Y_{y}\sbra{ F(\{ Y_{\beta^Y(y, \varepsilon, t) }  : t\geq 0 \}) \Big{|} T^Y(\varepsilon)<\infty}
\\
=&\bE_{h_S^{-1}(y)}\sbra{ F(\{ h_S (X_{\beta(y, \varepsilon, t)  })  : t\geq 
0 \}) \Big{|} T(\varepsilon)<\infty}
\\
=&n_{h_S^{-1}(y)}\sbra{ F(\{ h_S (X_{\tau(t) })  : t\geq 0\}) ; T(\varepsilon)<\infty}\Big{\slash}n_{h_S^{-1}(y)}\rbra{T(\varepsilon)<\infty},\label{8}
\end{align}
where 
$T^Y(\varepsilon):= T^{Y,-}_{y-\varepsilon}\land T^{Y,+}_{y+\varepsilon}$, 
$\beta^Y(y, \varepsilon, t):=\rbra{t+g^Y_{y} (T^Y(\varepsilon))}\land T_{y}^Y(T^Y(\varepsilon))$, 
$T(\varepsilon):= T^{-}_{h_S^{-1}(y-\varepsilon)}\land T^+_{h_S^{-1}(y+\varepsilon)}$ and 
$\beta(y, \varepsilon, t):=\tau\Big{(}t+ \tau^{-1}({g_{h_S^{-1}(y)}(T(\varepsilon))})\Big{)}\land T_{h_S^{-1}(y)}(T(\varepsilon))$. 
We also have
\begin{align}
n_{y}^Y\rbra{T^Y(\varepsilon)<\infty}=
-\log \bP^Y_{y}\rbra{L^{Y,x}_{T^Y(\varepsilon)} >1}
=&-\log \bP_{h_S^{-1}(y)}\rbra{H(y)L^{h_S^{-1}(y)}_{T(\varepsilon)} >1}\\
=&\frac{1}{H(y)}n_{h_S^{-1}(y)}\rbra{T(\varepsilon)<\infty}, \label{9}
\end{align}
where in the second equality, we used \eqref{12}. 
From \eqref{8} and \eqref{9}, we have, for $\varepsilon>0$ and non-negative measurable functional $F$, 
\begin{align}
&n_{y}^Y\sbra{ F(\{ Y_{t}  : t\geq 0 \}) ; T^Y(\varepsilon)<\infty }\\
=&\frac{n_{y}^Y\rbra{T^Y(\varepsilon)<\infty}}{n_{h_S^{-1}(y)} \rbra{T(\varepsilon)<\infty}}
n_{h_S^{-1}(y)}\sbra{ F(\{ h_S (X_{\tau(t) })  : t\geq 0\}) ; T(\varepsilon)<\infty}\\
=&\frac{1}{H(y)}
n_{h_S^{-1}(y)}\sbra{ F(\{ h_S (X_{\tau(t)})  : t\geq 0\}) ; T(\varepsilon)<\infty}, 
\end{align} 
and thus by taking the limit as $\varepsilon\downarrow 0$, we obtain \eqref{10}. 
\par 
(ii) We assume that $x\in h_S(I)$ is irregular for itself. 
By 
Assumption \ref{Cond202} (ii) and \eqref{14}, we have
\begin{align}
n^Y_y \sbra{ F(\{ Y_{t}  : t\geq 0 \})  }&= \frac{1}{H(y) l^{h_S^{-1}(x)}} \bP^{Y}_y\sbra{ F(\{Y_{t\land T^Y_y}: t\geq 0\})  }
\\
&=\frac{1}{H(y) l^{h_S^{-1}(y)}} \bP_{h_S^{-1}(y)}\sbra{ F(\{X_{\tau(t)\land T_{h_S^{-1}(y)}}: t\geq 0\})  }\\
&=\frac{1}{H(y)}n_{h_S^{-1}(y)}\sbra{ F(\{ h_S (X_{\tau(t)}) : t\geq 0 \})  }
. \label{7}
\end{align}
\par
The proof is complete. 
\end{proof}

\begin{proof}[Proof of Theorem \ref{Thm301}]
For $q\geq 0$ and $a \in h_S(I)$, from Theorem \ref{Thm204}, the function $y\mapsto W^{(q)}_Y(a, y)$ 
is an unique solution of the Volterra integral equation
\begin{align}
f(y)=W_Y(a, y)+q \int_{(x, a)} f(u) W_Y(y^\prime, y) m_Y(dy^\prime), \qquad y \in h_S(I), \label{5}
\end{align}
where $W_Y:=W_Y^{(0)}$. 
By Lemma \ref{LemA01}, we have, for $y, y^\prime\in h_S(I)$ with $y\leq y^\prime$, 
\begin{align}
W_Y( y^\prime, y)=\frac{1}{n^Y_{y}\sbra{T_{y^\prime}^{Y,+} <\infty}}
=&\frac{H(y)}{n_{h_S^{-1}(y)}\sbra{T_{h_S^{-1}(y^\prime)}^{+} <\infty}}\\
=&H(y) W_X(h_S^{-1}(y^\prime), h_S^{-1}(y)). \label{11}
\end{align} 
By \eqref{11}, 
the equation \eqref{5} is equivalent to the equation \eqref{4} and the proof is complete. 
\end{proof}

\section{Applications of Theorem \ref{Thm301}}\label{Sec004}
In this section, we give characterizations of the scale functions of positive or negative self-similar Markov processes with no positive jumps and continuous-state branching processes using the Volterra integral equations with the scale functions of spectrally negative L\'evy processes.

\subsection{Preliminaries: the scale functions of spectrally negative L\'evy processes}\label{Sec401}
In this section, we recall some facts about the scale functions of spectrally negative L\'evy processes. 
For the detail, see, e.g., \cite{KuzKypRiv2012} or \cite[Section 8]{Kyp2014}.
\par
Let $X$ be a spectrally negative L\'evy process. We define a function $\psi:[0, \infty) \to \bR$ as 
\begin{align}
\psi (\lambda):=\log \bE_0 \sbra{e^{\lambda X_1}}, \qquad \lambda \geq0. 
\end{align}
The function $\psi$ is called the Laplace exponent of $X$ and has the following form
\begin{align}
\psi (\lambda)= a\lambda +\frac{1}{2}\sigma^2 \lambda^2 +\int_{(-\infty, 0)}\rbra{e^{\lambda x}-1-\lambda x 1_{\{x>-1\}}}\Pi (dx)
,\quad \lambda \geq 0 ,
\end{align}
where $a\in\bR$, $\sigma\geq 0$ and $\Pi$ is a measure on $(-\infty , 0)$ which satisfies $\int_{(-\infty , 0)}(1\land x^2)\Pi(dx)<\infty$. 
Especially, when $X$ has bounded variation paths, i.e., $\sigma=0$ and $\int_{(-\infty , 0)}(1\land |x|)\Pi(dx)<\infty$ hold, 
we can write 
\begin{align}
\psi (\lambda)= \delta\lambda  +\int_{(-\infty, 0)}\rbra{e^{\lambda x}-1}\Pi (dx)
,\quad \lambda \geq 0 ,
\end{align}
with $\delta:=a - \int_{(-1, 0)}x\Pi(dx)$. 
For $q\geq0$, we define a function $W^{(q)}:\bR \to[0, \infty)$ such that $W^{(q)}=0$ on $(-\infty , 0)$, and $W^{(q)}$ on $[0, \infty)$ is continuous and satisfying
\begin{align}
\int_0^\infty e^{-\beta x} W^{(q)}(x) dx =\frac{1}{\psi(\beta)-q},\qquad \beta>\Phi(q), \label{17}
\end{align} 
where $\Phi(q):=\sup\{\lambda \geq 0 :\psi(\lambda)=q\}$. 
The function $W^{(q)}$ is generally called the $q$-scale function of $X$. 
\par
We take the reference measure $m$ of $X$ as the Lebesgue measure. 
We define local times ${\{L^x\}}_{x\in\bR}$ by \cite[Section V]{Ber1996} when $X$ has unbounded variation paths, or by \eqref{6} with $l^x=\frac{1}{\delta}$ when $X$ has bounded variation paths (these local times satisfy \eqref{2}). 
Then, the corresponding excursion measure $n_x$ of $X$ away from $x\in \bR$ satisfies 
\begin{align}
n_x \sbra{1- e^{-qT_x}}=\frac{1}{\Phi^\prime(q)}, \qquad q>0. 
\end{align}
Let $W^{(q)}_X$ be the $q$-scale function defined by the excursion measures above and \eqref{204}. 
Then, by \cite[Appendix A]{Nob2020_1}, we have 
\begin{align}
W^{(q)}(x-x^\prime)=W^{(q)}_X(x, x^\prime ),\qquad q\geq 0 , \ x, x^\prime \in \bR. \label{16}
\end{align}
The identity \eqref{16} means that for $q\geq 0$, 
the function $W^{(q)}$ defined by \eqref{17} can be regarded as the $q$-scale function defined by \eqref{204} with respect to the Lebesgue measure.
\par
In the following, we assume that the scale functions of the spectrally negative L\'evy processes are the functions ${\{W^{(q)}\}}_{q\geq 0}$ defined by \eqref{17}.

\subsection{Positive self-similar Markov processes with no positive jumps}\label{Sec402a}
Let $Y$ be a $(0, \infty)$-valued standard process. We assume that $Y$ satisfies the self-similarity, i.e., there exists a constant $\alpha>0$ such that for $x>0$ and $c>0$, the law of $\{cY_{c^{-\alpha}t}: t\geq 0\}$ under $\bP^Y_x$ is the same as  the law of $\{Y_{t}: t\geq 0\}$ under $\bP^Y_{cx}$. 
Then the process $Y$ is called a positive self-similar Markov process with index of self-similarity $\alpha$. 
We set $h_S(x)=e^x$ and $h_T(x)=e^{\alpha x}$ for $x\in\bR$. 
Then, by \cite[Theorem 13.1]{Kyp2014}, there exists a L\'evy process $X$ killed at an independent and exponentially distributed random time ${\bf{e}}_r$ with parameter $r\geq [0, \infty)$ such that the space-time changed process of $X$ as in Section \ref{SubSec203} is equivalent to $Y$. Here, we assume that ${\bf{e}}_0=\infty$ $\bP$-a.s.. 
\par
We assume that $Y$ has no positive jumps and no monotone paths, then the corresponding killed L\'evy process $X$ also has the same properties. 
We define the reference measure $m$ as the Lebesgue measure. 
Let $W^{(r)}$ denotes the $r$-scale function of $X^\prime$ which is $X$ without killing at ${\bf{e}}_r$. 
Note that we can regard $W^{(r)}$ as the $0$-scale function of $X$ 
since when we define the local times of $X$ and $X^\prime$ as in Section \ref{Sec401}, 
the $0$-scale function $W_{X}$ of $X$ in \eqref{3} satisfies $W_{X}=W^{(r)}_{X^\prime}$ for the $r$-scale functions $W^{(r)}_{X^\prime}$ of $X^\prime$ in \eqref{3} by the definition of the scale functions in Section \ref{Pre}. 
We fix a strictly positive measurable function $h_D$ on $( 0, \infty)$ and define $m_Y$ as the measure on $(0, \infty)$ such that $m_Y(B)=\int_Bh_D(y)m(h^{-1}_S(dy))= \int_B \frac{h_D(y)}{y} dy$ for $B\in\cB((0, \infty))$.  
We define the scale functions ${\{W_Y^{(q)}\}}_{q\geq 0 }$ in the same way as them in Section \ref{Sec003} from the local times of $X$ 
and the reference measure $m_Y$ on $(0, \infty)$. 
From Theorem \ref{Thm301}, 
for $q\geq 0$ and $a \in (0, \infty)$, the function $y \mapsto W^{(q)}_Y(a, y)$ on $(0, \infty)$ is an unique solution of the Volterra integral equation
\begin{align}
\begin{aligned}
f(y)=\frac{y^\alpha}{h_D(y)}&W^{(r)}( \log(a)-\log(y))\\
&+ \frac{y^\alpha}{h_D(y)} q\int_{(y, a)} f(y^\prime) W^{(r)}(\log (y^\prime)-\log (y))  \frac{h_D(y^\prime)}{y^\prime}dy^\prime, 
 \end{aligned}
 \quad y \in (0, \infty) .
 \label{15}
\end{align}
In particular, by letting $h_D(y)={y}$ for $y\in (0, \infty)$, 
the unique solution of \eqref{15} represent 
the scale functions ${\{W_Y^{(q)}\}}_{q\geq 0 }$ of $Y$ when the reference measure $m_Y(dy)=\frac{h_D(y)}{y}dy$ is the Lebesgue measure. 
\subsection{Negative self-similar Markov processes with no positive jumps}
We also write for the cases where $Y$ has the self-similarity and takes negative values. 
Let $Y$ be $(-\infty, 0)$-valued standard process with no positive jumps which has no monotone paths and satisfies the self-similarity with index of self-similarity $\alpha>0$. 
Then the process $Y$ is called a negative self-similar Markov process with index of self-similarity $\alpha$. 
In this case as well, by setting $h_S(x)=-e^{-x}$ and $h_T(x)=e^{-\alpha x}$, and by using \cite[Theorem 13.1]{Kyp2014}, $Y$ can be expressed as the space-time changed process in Section \ref{SubSec203} of a spectrally negative L\'evy process $X$ killed at an independent and exponentially distributed random time ${\bf{e}}_r$ with parameter $r\geq [0, \infty)$. 
We define the reference measure $m$ as the Lebesgue measure and the local times of $X$ as those for spectrally negative L\'evy processes in Section \ref{Sec401}. 
We write $W^{(r)}$ for the $r$-scale function of 
$X$ 
without killing with ${\bf{e}}_r$. 
Then, as in the cases of Section \ref{Sec402a}, the function $W^{(r)}$ can be regarded as the $0$-scale function of $X$. 
We fix a strictly positive measurable function $h_D$ on $(-\infty, 0)$ and define $m_Y$ as the measure on $(-\infty, 0)$ such that $m_Y(B)=\int_Bh_D(y)m(h^{-1}_S(dy))=\int_B \frac{h_D(y)}{-y} dy$ for $B\in\cB((-\infty, 0))$.  
We define the scale functions ${\{W^{(q)}_Y\}}_{q\geq 0}$ for $Y$ in the same way as for the cases of Section \ref{Sec402a}. 
Then by Theorem \ref{Thm301}, for $q\geq 0$ and $a \in (-\infty, 0)$, the function $y\mapsto W^{(q)}_Y(a, y)$ on $(-\infty , 0)$ 
is an unique solution of the Volterra integral equation 
\begin{align}
\begin{aligned}
f(y)=&\frac{{(-y)}^\alpha}{h_D(y)}W^{(r)}(\log(-y)-\log(-a))\\
&+ \frac{{(-y)}^\alpha}{h_D(y)} q\int_{(y, a)} f(y^\prime) W^{(r)}(\log (-y)-\log (-y^\prime))\frac{h_d(y^\prime)}{-y^\prime}dy^\prime,
\end{aligned}
\quad
y^\prime\in (-\infty , 0). 
\label{18}
\end{align}
In particular, by letting $h_D(y)=-{y}$ for $y\in ( -\infty, 0)$, 
the unique solution of \eqref{18} represent 
the scale functions ${\{W_Y^{(q)}\}}_{q\geq 0 }$ of $Y$ when the reference measure $m_Y(dy)=\frac{h_D(y)}{-y}dy$ is the Lebesgue measure. 

\subsection{Continuous-state branching processes}\label{Sec402}
I note that, in special cases, the results written in this section may not be entirely new by Remark \ref{Rem402}. 
\par
A $[0, \infty)$-valued standard process $Y$ 
which satisfies the branching property, i.e., 
\begin{align}
\bE^Y_{y+y^\prime}\sbra{e^{-\lambda Y_t}}=\bE^Y_{y}\sbra{e^{-\lambda Y_t}}\bE^Y_{y^\prime}\sbra{e^{-\lambda Y_t}}, \qquad
\lambda\geq0 , \  y,y^\prime \geq 0 , \ t\geq 0 ,
\end{align}
(for another definition of the branching property, see, e.g., \cite[Definition 1.14]{Kyp2014}), 
is known a continuous-state branching processes. 
Here, we assume that $0$ is the cemetery point of $Y$. 
Then, the process $-Y$ can be represented by the time change in Section \ref{SubSec203} of a spectrally negative L\'evy process $X$ killed on exiting $(-\infty , 0)$ with $h_T(x)=-\frac{1}{x}$ for $x\in(-\infty , 0)$ by \cite[Theorem 12.2]{Kyp2014} (the space change is not necessary, so we assume that $h_S(x)=x$ for $x\in(-\infty, 0)$). 
We define the reference measure $m$ as the Lebesgue measure and the local times of $X$ as those for spectrally negative L\'evy processes in Section \ref{Sec401}. 
We write $W$ for the $0$-scale function of $X$ without killed on exiting $(-\infty , 0)$. 
By the same argument as 
for the cases of Section \ref{Sec402a}, 
the function $W$ can be regarded as a $0$-scale function of $X$. 
We fix a strictly positive measurable function $h_D$ on $(-\infty, 0)$ and define $m_Y$ as the measure on $(-\infty, 0)$ such that $m_Y(B)=\int_B h_D(y) m(dy)=\int_B h_D(y) dy$ for $B\in\cB((-\infty, 0))$.  
We define the scale functions ${\{W^{(q)}_Y\}}_{q\geq 0}$ for $Y$ in the same way as in Section \ref{Sec402a}. 
Then by Theorem \ref{Thm301}, for $q\geq 0$ and $a \in (-\infty, 0)$, the function $y\mapsto W^{(q)}_Y(a, y)$ on $(-\infty , 0)$ 
is an unique solution of the Volterra integral equation 
\begin{align}
\begin{aligned}
f(y)=-\frac{1}{yh_D(y)}&W(a-y)\\
&-\frac{1}{yh_D(y)}q\int_{(y, a)} f(y^\prime) W(y^\prime-y)  h_D(y^\prime) dy^\prime, 
\end{aligned}
\qquad 
& y \in (-\infty, 0). \label{19}
\end{align}
In particular, by letting $h_D(y)=1$ for $y\in ( -\infty, 0)$, 
the unique solution of \eqref{19} represent 
the scale functions ${\{W_Y^{(q)}\}}_{q\geq 0 }$ of $Y$ when the reference measure $m_Y(dy)= h_D(y) dy$ is the Lebesgue measure. 
\begin{Rem}\label{Rem402}
Before I 
started this work, 
Dr. Kosuke Yamato mentioned that if the Laplace exponent of the corresponding spectrally negative L\'evy process satisfies $\psi^\prime(0+)>0$, the characterization of the scale functions of $-Y$ could probably be obtained 
by using \cite[Proposition 3.3 and Theorem 3.5]{NobYam2023+}, the facts that the $0$-scale functions of a continuous-state branching process and the corresponding spectrally negative L\'evy process are equal 
up to multiplication by a constant, and the comparison quasi-stationary distributions \cite[Theorem 5.2.]{NobYam2023+} and \cite[Theorem 3.1]{Lam2007}. 
I also think that the above formula can be obtained using this method. 
However, 
this idea is not mentioned in any paper, 
and after discussing with him, I wrote about it here. 
\end{Rem}

\section*{Acknowledgments}
The auther was supported by JSPS KAKENHI grant no. JP21K13807 and JSPS Open Partnership Joint Research Projects grant no. JPJSBP120209921. In addition, the author stayed at Centro de Investigaci\'on en Matem\'aticas in Mexico as a JSPS Overseas Research Fellow and received support regarding the research environment there. The author was grateful for their support during his visit. 
I am grateful to Dr. Kosuke Yamato for allowing me to write the results concerning continuous-state branching processes in Section \ref{Sec402}. 
I am also grateful to Prof. Kouji Yano for giving me some advice on the structure of this paper.




\bibliographystyle{jplain}
\bibliography{NOBA_references_05}

\begin{thebibliography}{10}

\bibitem{AvrPalPis2007}
F.~Avram, Z.~Palmowski, and M.~R. Pistorius.
\newblock On the optimal dividend problem for a spectrally negative {L}\'evy
  process.
\newblock {\em Ann. Appl. Probab.}, Vol.~17, No.~1, pp. 156--180, 2007.

\bibitem{Ber1996}
J.~Bertoin.
\newblock {\em {L}\'evy processes}, Vol. 121 of {\em Cambridge Tracts in
  Mathematics}.
\newblock Cambridge University Press, Cambridge, 1996.

\bibitem{Ber1997}
J.~Bertoin.
\newblock Exponential decay and ergodicity of completely asymmetric {L}\'{e}vy
  processes in a finite interval.
\newblock {\em Ann. Appl. Probab.}, Vol.~7, No.~1, pp. 156--169, 1997.

\bibitem{Blu1992}
R.~M. Blumenthal.
\newblock {\em Excursions of {M}arkov processes}.
\newblock Probability and its Applications. Birkh\"{a}user Boston, Inc.,
  Boston, MA, 1992.

\bibitem{BluGet1968}
R.~M. Blumenthal and R.~K. Getoor.
\newblock {\em Markov processes and potential theory}.
\newblock Pure and Applied Mathematics, Vol. 29. Academic Press, New
  York-London, 1968.

\bibitem{CzaPerRolYam2019}
I.~Czarna, J.~L. P\'{e}rez, T.~Rolski, and K.~Yamazaki.
\newblock Fluctuation theory for level-dependent {L}\'{e}vy risk processes.
\newblock {\em Stochastic Process. Appl.}, Vol. 129, No.~12, pp. 5406--5449,
  2019.

\bibitem{GemHor1980}
D.~Geman and J.~Horowitz.
\newblock Occupation densities.
\newblock {\em Ann. Probab.}, Vol.~8, No.~1, pp. 1--67, 1980.

\bibitem{KuzKypRiv2012}
A.~Kuznetsov, A.~E. Kyprianou, and V.~Rivero.
\newblock The theory of scale functions for spectrally negative {L}\'evy
  processes.
\newblock In {\em {L}\'evy matters {II}}, Vol. 2061 of {\em Lecture Notes in
  Math.}, pp. 97--186. Springer, Heidelberg, 2012.

\bibitem{Kyp2014}
A.~E. Kyprianou.
\newblock {\em Fluctuations of {L}\'evy processes with applications}.
\newblock Universitext. Springer, Heidelberg, second edition, 2014.
\newblock Introductory lectures.

\bibitem{KypLoe2010}
A.~E. Kyprianou and R.~L. Loeffen.
\newblock Refracted {L}\'evy processes.
\newblock {\em Ann. Inst. Henri Poincar\'e Probab. Stat.}, Vol.~46, No.~1, pp.
  24--44, 2010.

\bibitem{Lam2007}
A.~Lambert.
\newblock Quasi-stationary distributions and the continuous-state branching
  process conditioned to be never extinct.
\newblock {\em Electron. J. Probab.}, Vol.~12, pp. no. 14, 420--446, 2007.

\bibitem{LiPal2018}
B.~Li and Z.~Palmowski.
\newblock Fluctuations of omega-killed spectrally negative {L}\'{e}vy
  processes.
\newblock {\em Stochastic Process. Appl.}, Vol. 128, No.~10, pp. 3273--3299,
  2018.

\bibitem{Nob2020_1}
K.~Noba.
\newblock Generalized scale functions of standard processes with no positive
  jumps.
\newblock {\em Electron. Commun. Probab.}, Vol.~25, pp. Paper No. 8, 12, 2020.

\bibitem{NobYam2023+}
K.~Noba and K.~Yamato.
\newblock Analytic property of generalized scale functions for standard
  processes with no negative jumps and its application to quasi-stationary
  distributions.
\newblock {\em arXiv preprint arXiv:2308.09935}, 2023.

\bibitem{NobYan2019}
K.~Noba and K.~Yano.
\newblock Generalized refracted {L}\'{e}vy process and its application to exit
  problem.
\newblock {\em Stochastic Process. Appl.}, Vol. 129, No.~5, pp. 1697--1725,
  2019.

\bibitem{Pat2009}
P.~Patie.
\newblock Infinite divisibility of solutions to some self-similar
  integro-differential equations and exponential functionals of {L}\'{e}vy
  processes.
\newblock {\em Ann. Inst. Henri Poincar\'{e} Probab. Stat.}, Vol.~45, No.~3,
  pp. 667--684, 2009.

\end{thebibliography}

\end{document}